\documentclass[10pt]{amsart}
\usepackage{txfonts}
\usepackage{bbm}
\usepackage{amsfonts}
\usepackage{indentfirst}
\usepackage{amsmath}
\usepackage[all, cmtip, line]{xy}
\usepackage{mathrsfs}
\usepackage{amssymb}

\newtheorem{prop}{Proposition}[section]
\newtheorem{theo}[prop]{Theorem}
\newtheorem{lem}[prop]{Lemma}

\newtheorem{cor}[prop]{Corollary}
\newtheorem{defi}[prop]{Definition}
\newcommand{\be}{\begin{equation}}
\newcommand{\ee}{\end{equation}}
\newcommand{\bes}{\begin{eqnarray}}
\newcommand{\ees}{\end{eqnarray}}
\newcommand{\bess}{\begin{eqnarray*}}
\newcommand{\eess}{\end{eqnarray*}}

\allowdisplaybreaks

\begin{document}
\setlength{\baselineskip}{16pt}
\pagestyle{myheadings}

\title[Representations of  Quantum Double]{}
\author[J.C. Dong, H.X. Chen ]{}
\date{}
\begin{center}
\textbf{\large{The Representations of  Quantum Double of Dihedral
Groups}\footnote[1]{The work is supported by NSF of China
(10771183), and also supported by Doctorate Foundation
(200811170001), Ministry of Education of China;
 Agricultural Machinery Bureau Foundation of Jiangsu Province(GXZ08001).}}\\
\vskip5pt
 \textbf{Jingcheng Dong}\\\textit{College of Engineering, Nanjing
Agricultural University\\ Nanjing Jiangsu, 210031, China\\Department
of Mathematics, Yangzhou University\\ Yangzhou
Jiangsu, 225002, China }\\\textit{E-mail: dongjc@njau.edu.cn}\\
\textbf{Huixiang Chen\footnote[2]{Corresponding author.}}\\
\textit{Department of Mathematics, Yangzhou University\\ Yangzhou
Jiangsu, 225002, China\\E-mail: hxchen@yzu.edu.cn}\\
\vskip5pt Received 22 March 2009\\
Revised 22 September 2009\\
\vskip5pt Communicated by Quanshui Wu
\end{center}

\maketitle

\textbf{Abstract.}\quad Let $k$ be an algebraically closed field of
odd characteristic $p$, and let $D_n$ be the Dihedral group of order
$2n$ such that $p\mid 2n$. Let $D(kD_n)$ denote the quantum double
of the group algebra $kD_n$. In this paper, we describe the
structures of all finite dimensional indecomposable left
$D(kD_n)$-modules, equivalently, of all finite dimensional
indecomposable Yetter-Drinfeld $kD_n$-modules, and classify them.



\section{Introduction}
Quantum groups arose from the study of quantum Yang-Baxter
equations. Quantum groups are defined in terms of what Drinfeld
calls ``quasitriangular Hopf algebra"  and their construction is
based on a general procedure also due to Drinfeld \cite{drinfeld}
assigning to a finite dimensional Hopf algebra $H$ a quasitriangular
Hopf algebra $D(H)$. The Hopf algebra $D(H)$ is called the quantum
double of $H$, or the Drinfeld double of $H$. It has brought
remarkable applications to new aspects of representation theory,
noncommutative geometry, low-dimensional topology and so on.

The quantum double $D(kG)$ of a finite dimensional group algebra
$kG$ has attracted many mathematicians' interest recently. Let $k$
be an algebraically closed field of positive characteristic $p$. If
the order of $G$ is divisible by $p$, then $D(kG)$ is not semisimple
by an analogue of Maschke's theorem (see \cite{witherspoon}). In
this case, we need to find all indecomposable $D(kG)$-modules in
order to describe the structure of the representation ring (or Green
ring) of $D(kG)$.

In this paper, we investigate the representations of $D(kD_n)$ for
the Dihedral group $D_n$. In section 2, we recall some basic results
on representation theory of finite groups and construct
Yetter-Drinfeld modules over group algebras, and make some
preparations for the rest of this paper.

In section 3, we discuss the indecomposable representations of some finite
groups. Suppose that $K_4$ is the Klein four group. Then
$kK_4$ is semisimple if and only if the characteristic of $k$ is not equal to $2$.
In this case, $kK_4$ has only 4 simple
modules up to isomorphism. Let $k$ be an algebraically closed field of odd
characteristic $p$. Suppose that $C_n$ is a cyclic group of order
$n=p^st$ with $(p, t)=1$. Then $kC_n$ has exactly $n$ indecomposable
modules with dimension $1, 2, \cdots, p^s$, respectively. Moreover,
there are $t$ indecomposable modules of each dimension. These
modules are obtained as indecomposable summands of the modules
induced from modules over a Sylow $p$-subgroup of $C_n$. Suppose that
$D_n$ is the Dihedral group of order $2n$ with $n=p^st$ as above. If $n$
is odd, then $kD_n$ has exactly $\frac{t+3}{2}p^s$ indecomposable
modules. If $n=p^st$ is even, then $kD_n$ has exactly
$\frac{t+6}{2}p^s$ indecomposable modules. These modules are
obtained as indecomposable summands of the modules induced from
those over a normal subgroup.

We state the main results in section 4. Following the
characterization of Yetter-Drinfeld $kG$-modules in
\cite{witherspoon, mason},  we construct all finite dimensional indecomposable
Yetter-Drinfeld $kD_n$-modules, where $D_n$ is the Dihedral group of order $2n$.
These modules are induced from the indecomposable $kC_{D_n}(g)$-modules
for some $g\in D_n$.

\section{Preliminaries}
Throughout this paper, we work over an algebraically closed field of odd
characteristic $p$. Unless otherwise stated, all modules are left modules, all comodules are right
comodules, and moreover they are finite dimensional over $k$.
$\otimes$ means $\otimes_k$.

There is an important way of constructing modules over a
group from modules over its subgroups, which was originally
described by Frobenius in \cite{frobenius}. Let $H$ be a subgroup of
a finite group $G$. If $N$ is a $kH$-module, then we have an induced
$kG$-module $N\!\!\uparrow^G=kG\otimes_{kH}N$. Since $kG$ is free as
a right $kH$-module, of rank $[G:H]$, one can write
$N\!\!\uparrow^G=\bigoplus_{g_i}g_i\otimes N$ as a sum of $k$-vector spaces,
where the sum runs over a set of left coset representatives of $H$
in $G$. The action of $G$ is given by $g(g_i\otimes n)=gg_i\otimes
n=g_j\otimes hn$, where $g_j$ is the left coset representative such
that $gg_i=g_jh$ with $h\in H$. The representation of $G$ corresponding
to the $kG$-module $N\!\!\uparrow^G$ is called an induced representation. In subsequent, we will
use the following terminologies of representations and modules.

Let $\{g_1, g_2, \cdots, g_t\}$ be a set of left coset
representatives of $H$ in $G$, and $\{v_1, \cdots, v_r\}$ be a
$k$-basis of $kH$-module $N$. Let $\rho$ be the matrix
representation of $H$ corresponding to $N$ with respect to the given
basis. Then, with respect to the $k$-basis $\{g_i\otimes v_j|1\leq
i\leq t, 1\leq j\leq r\}$, the induced matrix representation
$\Omega$ corresponding to $N\!\!\uparrow^G$ can be described as
follows
$$\Omega(g)=(\rho(g_j^{-1}gg_i))_{t\times t},$$ where $\rho$ is extended to $G$ by setting
$\rho(x)=0$ for $x\in G\backslash H$. Thus $\Omega(g)$ is
partitioned into a $t\times t$ array of $r\times r$ blocks.

\begin{theo}\cite{higman}\label{higman}
Let $kG$ be the group algebra of a finite group $G$. Then $kG$ is of
finite representation type if and only if $G$ has cyclic Sylow
$p$-subgroups.
\end{theo}

Suppose that a Sylow $p$-subgroup $P$ of $G$ is cyclic. Then we have
the following theorem \cite{berman} for the number of non-isomorphic
indecomposable modules over $kG$.

\begin{theo}\cite{berman}\label{berman}
Suppose that a Sylow $p$-subgroup $P$ of  a finite group $G$ is cyclic. If $P\lhd
G$, then the number of non-isomorphic indecomposable modules of $kG$
is equal to $|P|r$, where $r$ is the number of $p$-regular conjugate
classes of $G$.
\end{theo}

Let $x\in G$ and $H$ be a subgroup of $G$. Then ${}^x\!H=xHx^{-1}$
is also a subgroup of $G$. For any $kH$-module $V$, there is a
$k[{}^x\!H]$-module ${}^xV=x\otimes V=\{x\otimes v|v\in V\}$
with the action given by
${}^x\!y(x\otimes v)=x\otimes yv$ for all $y\in H$, where
${}^x\!y=xyx^{-1}$.

\begin{theo}[Mackey Decomposition Theorem]\label{Mac}
Let $H, K$ be subgroups of a finite group $G$ and $W$ be a $kK$-module.  Then
$$W\!\!\uparrow^G\!\!\downarrow_H\cong\bigoplus_{HgK}({}^gW)
\!\!\downarrow_{H\cap{}^gK}\uparrow^H, $$ where the sum runs over
the  double cosets  of $H$ and $K$ in $G$.
\end{theo}

Suppose that $H\lhd G$ and $V$ is a $kH$-module.  The inertia group
$T_H(V)=T(V)$ of $V$ is defined by $T(V)=\{x\in G|{}^xV\cong V\}$.
Clearly,  $T(V)$ is a subgroup of $G$ and $H\subseteq T(V)$.

\begin{theo}\cite{Feit}\label{T(W)}
Suppose that $G$ is a finite group, $H\lhd G$ and $W$ is an indecomposable $kH$-module such
that $T(W)=H$. Then $W\!\!\uparrow^G$ is indecomposable.
\end{theo}

\begin{theo}\cite{conlon}\label{conlon}
Suppose that a finite group $G$ is a direct product of two normal
subgroups $G_1$ and $G_2$. Then every indecomposable $kG$-modules is
of the form $V_1\otimes V_2$ for two indecomposable $kG_i$-modules
$V_i$, $i=1, 2$ if and only if at least one of the orders $|G_1|$ and
$|G_2|$ is not divisible by $p$.
\end{theo}

Let $G$ be a finite group. For any $g\in G$, define $\phi_g\in
(kG)^{*}$ by $\phi_g(h)=\delta_{g,h},\ h\in G.$ Then the quantum
double $D(kG)$ of the group algebra $kG$ has a $k$-basis
$\{\phi_g\otimes h|g, h\in G\}$.  In this case, the multiplication
is given by $$(\phi_g\otimes h)(\phi_{g'}\otimes
h')=\phi_g\phi_{hg'h^{-1}}\otimes
 hh'=\delta_{g, hg'h^{-1}}\phi_g\otimes hh'.$$ Thus the
identity is $1_{D(kG)}=\sum_{g\in G}\phi_g\otimes 1$, where $1$ is
the identity of $G$. The comutiplication $\Delta$, the counit
$\varepsilon$ and the antipode $S$ are given by
$$\begin{array}{c}
\Delta(\phi_g\otimes h)=\sum_{x\in G}
(\phi_x\otimes h)\otimes(\phi_{gx^{-1}}\otimes h),\\
\varepsilon(\phi_g\otimes h)=\delta_{1, g},\ S(\phi_g\otimes
h)=\phi_{h^{-1}g^{-1}h}\otimes h^{-1},\\
\end{array}$$ where $g, h\in G$. For the definition of quantum double $D(H)$
of any finite dimensional Hopf algebra $H$, the reader is directed
to \cite[Chapter IX. 4]{kassel}.

\begin{defi}
For a Hopf algebra $H$ with a bijective antipode $S$, a
Yetter-Drinfeld $H$-module $M$ is both a left $H$-module and a right
$H$-comodule satisfying the following two equivalent compatibility
conditions
$$\begin{array}{c}
\sum h_1\cdot m_0\otimes h_2m_1=\sum (h_2\cdot m)_0\otimes (h_2\cdot
m)_1h_1,\\
\sum(h\cdot m)_0\otimes (h\cdot m)_1=\sum h_2\cdot m_0\otimes
h_3m_1S^{-1}(h_1),\\
\end{array}$$ for all $h\in H$ and $m\in M$.
\end{defi}
The category of Yetter-Drinfeld $H$-modules is denoted by
${}_H\mathcal{YD}^H$. The following theorem gives the connection
between the $D(H)$-modules and the Yetter-Drinfeld $H$-modules.

\begin{theo}\cite{majid}
Let $H$ be a finite dimensional Hopf algebra. Then the
category ${}_H\mathcal{YD}^H$ of Yetter-Drinfeld $H$-modules can be identified with
the category ${}_{D(H)}\mathcal{M}$ of left modules over the quantum
double $D(H)$.
\end{theo}

Hence, we just need to study the Yetter-Drinfeld modules in this
paper.

Now let $H=kG$ be a finite dimensional group algebra. Let
$\mathcal{K}(G)$ be the set of conjugate classes of $G$. For any
$g\in G$, let $C_G(g)=\{x\in G|xg=gx\}$ be the centralizer of $g$ in
$G$. For any $C\in \mathcal{K}(G)$, fix a $g_C\in C$. Then
$\{g_C|C\in \mathcal{K}(G)\}$ is a set of representatives of
conjugate classes of $G$. Now let $N$ be a $kC_G(g_C)$-module. Then
$N\!\!\uparrow^G=kG\otimes_{kC_G(g_C)}N$ is a $kG$-module. Define a
$k$-linear map $\varphi:N\!\!\uparrow^G\rightarrow
N\!\!\uparrow^G\otimes kG$ by $\varphi(g\otimes n)=(g\otimes
n)\otimes gg_Cg^{-1}$ for all $g\in G$ and $n\in N$. Clearly, we
have:

\begin{lem}
Let $C\in \mathcal{K}(G)$ and $N$ be a $kC_G(g_C)$-module. Then
$(N\!\!\uparrow^G, \varphi)$ is a Yetter-Drinfeld $kG$-module,
denoted by $D(N)$.
\end{lem}

Let $M$ be a Yetter-Drinfeld $kG$-module with coaction $\varphi:
M\rightarrow M\otimes kG$.  For $C\in \mathcal{K}(C)$, let
$$M_C=\bigoplus_{g\in C}M_g$$
where $M_g=\{m\in M|\varphi(m)=m\otimes g\}$. An easy computation
shows that $M_g$ is a $kC_G(g)$-submodule of
$M\!\!\downarrow_{C_G(g)}$ and $M_C$ is a Yetter-Drinfeld
$kG$-submodule of $M$.

By \cite{witherspoon, mason}, we have the following characterization
of Yetter-Drinfeld $kG$-modules.

\begin{theo}\label{main1}
Let $C\in \mathcal{K}(G)$ and $N$ be a $kC_G(g_C)$-module. Then
$D(N)$ is indecomposable (resp., simple) Yetter-Drinfeld $kG$-module
if and only if $N$ is indecomposable (resp., simple)
$kC_G(g_C)$-module.
\end{theo}

\begin{theo}\label{main2}
Let $M$ be an indecomposable (resp., simple) Yetter-Drinfeld
$kG$-module. Then there exists a conjugate class $C\in
\mathcal{K}(G)$ such that $M=M_C\cong D(M_{g_C})$.
\end{theo}

\begin{cor}\label{main3}
Let $N_1$ and $N_2$ be indecomposable (resp., simple)
$kC_G(g_C)$-module. Then $D(N_1)\cong D(N_2)$ if and only if
$N_1\cong N_2$.
\end{cor}

Thus, up to isomorphism, there is a one-one correspondence between
the indecomposable (resp., simple) $kC_G(g_C)$-modules and
indecomposable (resp., simple) Yetter-Drinfeld $kG$-modules.

\section{ The indecomposable representations of some finite groups}

In this section, we discuss the representations of Klein four group,
cyclic groups and dihedral groups.

Let $K_4=\langle 1, a, b, ab|a^2=b^2=1,(ab)^2=1\rangle$ be the Klein four group.
Then $K_4$ is a direct product of $2$ cyclic groups of order $2$,
and $kK_4$ is a semisimple algebra. Obviously, $kK_4$ has exactly $4$
irreducible representations, and all of them are of degree one. They can be described as follows:
\begin{equation*}
\begin{cases}\rho_1(a)=1,   \\ \rho_1(b)=1;
\end{cases}
\begin{cases}\rho_2(a)=1,   \\ \rho_2(b)=-1;
\end{cases}
\begin{cases}\rho_3(a)=-1,   \\ \rho_3(b)=1;
\end{cases}
\begin{cases}\rho_4(a)=-1,   \\ \rho_4(b)=-1.
\end{cases}
\end{equation*}

Let $C_n=\{1, g, \cdots, g^{n-1}\}$ be a cyclic group of order $n$.
If $p\nmid n$, then $kC_n$ is semisimple. In this case, $kC_n$ has
$n$ irreducible representations, and all of them are of degree $1$.
Furthermore, they can be described as $\rho_i(g)=\xi^{i-1},
\mbox{for\,} 1\leq i\leq n,$ where $\xi\in k$ is an $n$-th primitive
root of unity.

From now on, assume that $n=p^st$ with $p\nmid t$ and that $\xi$ is
a $t$-th primitive root of unity in $k$.

\begin{lem}
$kC_n$ is of finite representation type.
\end{lem}
\begin{proof}
It follows from Theorem \ref{higman} and the fact that the subgroups
of a cyclic group are also cyclic.
\end{proof}

\begin{lem}\label{number}
$kC_n$ has $n$ non-isomorphic indecomposable modules.
\end{lem}
\begin{proof}
Note that $n=p^st$. It is easy to check that the number of $p$-regular conjugate classes
of $C_n$ is $t$. Since the Sylow $p$-subgroups of $C_n$ are normal, the
lemma then follows from Theorem \ref{berman}.
\end{proof}

The following lemma is well known in the representation theory of
finite groups.

\begin{lem}\label{lemwellknown}
Let $P=\{1, g, \cdots, g^{p^s-1}\}$ be a cyclic group of order
$p^s$. Then  the $p^s$ non-isomorphic indecomposable
representations of $kP$ can be represented by the Jordan matrices
\begin{center}
 $\rho_r(g)=\left(\begin{tabular}{cccccc}
$1$&$1$&$0$&$\cdots$&$0$&$0$\\
$0$&$1$&$1$&$\cdots$&$0$&$0$\\
$0$&$0$&$1$&$\cdots$&$0$&$0$\\
$\cdots$&$\cdots$&$\cdots$&$\cdots$&$\cdots$&$\cdots$\\
$0$&$0$&$0$&$\cdots$&$1$&$1$\\
$0$&$0$&$0$&$\cdots$&$0$&$1$\\
\end{tabular}\right)_{r\times r}\ ,$
\end{center}
where $1\leqslant r \leqslant p^s$.
\end{lem}

Theoretically, there is a perfect method to describe the
indecomposable representations of an arbitrary cyclic group. Let
$C_n=\langle a\rangle$ be a cyclic group of order $n$.
Then $C_n$ is a direct product of $2$ cyclic groups
$C_{p^s}=\langle a^t\rangle$ and $C_{t}=\langle a^{p^s}\rangle$. The
indecomposable matrix representations of $C_{p^s}$ and $C_{t}$ are
well described. Following Theorem \ref{conlon}, the indecomposable
matrix representations of $C_n$ can be described too. However, the
description depends on finding the solution of
equation $tx+p^sy=p^st+1$, where $1\leqslant x\leqslant p^s$, $1\leqslant y\leqslant
t-1$. Hence, we will use a transparent method instead of what we
mentioned above.

\begin{theo}\label{prrep}
Let $C_n=\langle a \rangle$ be a cyclic group of order
$n$. Then the $n$ non-isomorphic indecomposable
representations of $kC_n$ can be described as follows:
\begin{center}
 $\rho_{r, i}(a)=\left(\begin{tabular}{cccccc}
$\xi^i$&$1$&$0$&$\cdots$&$0$&$0$\\
$0$&$\xi^i$&$1$&$\cdots$&$0$&$0$\\
$0$&$0$&$\xi^i$&$\cdots$&$0$&$0$\\
$\cdots$&$\cdots$&$\cdots$&$\cdots$&$\cdots$&$\cdots$\\
$0$&$0$&$0$&$\cdots$&$\xi^i$&$1$\\
$0$&$0$&$0$&$\cdots$&$0$&$\xi^i$\\
 \end{tabular}\right)_{r\times r}\ ,$
\end{center}
where $1\leqslant r \leqslant p^s$ and $0\leqslant i \leqslant t-1$.
\end{theo}

\begin{proof}
Obviously, $P=\{1, a^t, a^{2t}, \cdots, a^{(p^s-1)t}\}$ is a Sylow
$p$-subgroup and we can choose $\{1, a, a^2, \cdots, a^{t-1}\}$ as a
set of the coset representatives of $P$ in $G$.

Following Lemma \ref{lemwellknown}, $kP$ has exactly $p^s$
non-isomorphic indecomposable representations, which are
represented by the Jordan matrices
\begin{center}
  $A_r:=\rho_r(a^t)=\left(\begin{tabular}{cccccc}
$1$&$1$&$0$&$\cdots$&$0$&$0$\\
$0$&$1$&$1$&$\cdots$&$0$&$0$\\
$0$&$0$&$1$&$\cdots$&$0$&$0$\\
$\cdots$&$\cdots$&$\cdots$&$\cdots$&$\cdots$&$\cdots$\\
$0$&$0$&$0$&$\cdots$&$1$&$1$\\
$0$&$0$&$0$&$\cdots$&$0$&$1$\\
 \end{tabular}\right)_{r\times r}$\ , $1\leqslant r\leqslant p^s$.
\end{center}

For $1\leqslant r\leqslant p^s$, let $\Omega_{r, t}$ be the matrix
representation of $kC_n$ induced from $\rho_r$. Then we have

\begin{center}
\begin{tabular}{ccl}
 $\Omega_{r, t}(a)$&=&$\left(\begin{tabular}{cccc}
$\rho_r(1\cdot a\cdot1)$&$\rho_r(1\cdot a\cdot a)$ &$\cdots$ &$\rho_r(1\cdot a\cdot a^{t-1})$\\
$\rho_r(a^{-1}\cdot a\cdot1)$&$\rho_r(a^{-1}\cdot a\cdot a)$ &$\cdots$ &$\rho_r(a^{-1}\cdot a\cdot a^{t-1})$\\
$\cdots$&$\cdots$&$\cdots$&$\cdots$\\
$\rho_r(a^{-(t-1)}\cdot a\cdot1)$&$\rho_r(a^{-(t-1)}\cdot a\cdot a)$ &$\cdots$& $\rho_r(a^{-(t-1)}\cdot a\cdot a^{t-1})$\\
 \end{tabular}\right)$\\
 &=&$\left(\begin{tabular}{cccccc}
$0$      &   $0$      &0&$\cdots$   &$0$      &$A_r$\\
$I_r$    &   $0$      &0&$\cdots$   &$0$      &$0$\\
$0$      &   $I_r$    &0&$\cdots$   &$0$      &$0$\\
$\cdots$ &$\cdots$&$\cdots$&$\cdots$&$\cdots$&$\cdots$\\
$0$      &   $0$&0          &$\cdots$   &$0$      &$0$\\
$0$      &   $0$&0          &$\cdots$   &$I_r$    &$0$\\
\end{tabular}\right)$\ ,
\end{tabular}
\end{center}
where $I_r$ is an $r\times r$ identity matrix. We write
$B=\Omega_{r, t}(a)$.

Now we compute the Jordan canonical form of matrix $B$. To do so, we
have to first determine the characteristic polynomial of the matrix $B$.

Since $|\lambda I-B|=(\lambda-1)^r(\lambda-\xi)^r
(\lambda-\xi^2)^r\cdots(\lambda-\xi^{t-1})^r$, the matrix $B$ has
$t$ distinct eigenvalues $1, \xi, \xi^2, \cdots, \xi^{t-1}$. It
follows that the Jordan canonical form of $B$ contains $t$ Jordan
segments, and each segment is composed of Jordan blocks
corresponding to the same eigenvalue. It is obvious that every
Jordan segment is an $r\times r$-matrix.

By Lemma \ref{number}, we know that every Jordan segment contains
only one Jordan block, which has the following form $\rho_{r,i}(a)$,
where $0\leqslant i \leqslant t-1$. This completes the proof.
\end{proof}

Let $D_n=\langle a,b|a^n=1, b^2=1, (ab)^2=1\rangle=\{1, a, \cdots,
a^{n-1}, b, ba, \cdots, ba^{n-1}\}$ be the Dihedral group of order
$2n$. Then the Sylow $p$-subgroups of $D_n$ have order $p^s$. The
cyclic subgroup $C_n=\{1, a, \cdots, a^{n-1}\}$ is a normal subgroup
of $D_n$. Let $P$ be the subgroup of $C_n$ generated by $a^t$. Then
$P=\{1, a^t, \cdots, a^{(p^s-1)t}\}$ is a normal Sylow $p$-subgroup
of $D_n$.

Suppose that $n=p^st$ is odd. Then the conjugate classes of $D_n$
are $\{1\},\ \{a^i, a^{n-i}\}\ (1\leqslant
i\leqslant\frac{n-1}{2}),\ \{a^jb|0\leqslant j\leqslant n-1\}$ with
representatives $1,\ a^i\ (1\leqslant i\leqslant\frac{n-1}{2}),\ b,$
where $1$, $b$, $a^{p^s}$, $a^{2p^s}$, $\cdots$,
$a^{\frac{t-1}{2}p^s}$ are $p$-regular.

Suppose that $n=p^st$ is even. Then the conjugate classes of $D_n$
are $\{1\},\ \{a^i, a^{n-i}\}\ (1\leqslant i\leqslant\frac{n}{2}),\
\{a^{2i}b|0\leq i \leq\frac{n-2}{2}\},\ \{a^{2i+1}b|0\leqslant i
\leqslant\frac{n-2}{2}\}$ with representatives $1,\ a^i\ (1\leqslant
i\leqslant\frac{n}{2}),\ b,\ ab,$ where $1, b, ab, a^{p^s},
a^{2p^s}, \cdots, a^{\frac{t}{2}p^s}$ are $p$-regular.

Thus by Theorem \ref{higman} and Theorem \ref{berman},
we have the following theorem.

\begin{theo}
The group algebra $kD_n$ is of finite representation type. If
$n=p^st$ is odd, then $kD_n$ has $\frac{t+3}{2}p^s$ indecomposable
modules. If $n=p^st$ is even, then $kD_n$ has $\frac{t+6}{2}p^s$
indecomposable modules.
\end{theo}

By Theorem \ref{prrep}, $kC_n$ has $n$ indecomposable
representations $A_{r,i}:=\rho_{r, i}(a)$, where $1\leqslant
r\leqslant p^s$ and $0\leqslant i\leqslant t-1$. Since $\{1, b\}$ is
a set of left coset representatives of $C_n$ in $D_n$, the induced
representation $\Omega_{2r, i}:=\rho_{r, i}\!\!\uparrow^{D_n}$ is
given by
\begin{center}
\begin{tabular}{cc}
 $\Omega_{2r, i}(a)=\left(\begin{tabular}{cccc}
$A_{r, i}$&$0$\\
$0$&$A_{r, i}^{-1}$
 \end{tabular}\right)_{2r\times 2r}, $&
 $\Omega_{2r, i}(b)=\left(\begin{tabular}{cccc}
$0$&$I_r$\\
$I_r$&$0$
 \end{tabular}\right)_{2r\times 2r}$\ .
 \end{tabular}
\end{center}

\begin{lem}
With the notations above, the inverse $A_{r, i}^{-1}$ of $A_{r, i}$
is equal to
\begin{center}
\begin{tabular}{c}
 $\left(\begin{tabular}{ccccccc}
$\xi^{-i}$&$-\xi^{-2i}$&$\xi^{-3i}$&$\dots$&
$(-1)^{r-3}\xi^{-(r-2)i}$&$(-1)^{r-2}\xi^{-(r-1)i}$&$(-1)^{r-1}\xi^{-ri}$\\
$0$&$\xi^{-i}$&$-\xi^{-2i}$&$\dots$&
$(-1)^{r-4}\xi^{-(r-3)i}$&$(-1)^{r-3}\xi^{-(r-2)i}$&$(-1)^{r-2}\xi^{-(r-1)i}$\\
$0$&$0$&$\xi^{-i}$&$\dots$&
$(-1)^{r-5}\xi^{-(r-4)i}$&$(-1)^{r-4}\xi^{-(r-3)i}$&$(-1)^{r-3}\xi^{-(r-2)i}$\\
$\cdots$&$\cdots$&$\cdots$&$\cdots$&$\cdots$&$\cdots$&$\cdots$\\
$0$&$0$&$0$&$\cdots$&$\xi^{-i}$&$-\xi^{-2i}$&$\xi^{-3i}$\\
$0$&$0$&$0$&$\cdots$&$0$&$\xi^{-i}$&$-\xi^{-2i}$\\
$0$&$0$&$0$&$\cdots$&$0$&$0$&$\xi^{-i}$
 \end{tabular}\right)$
 \end{tabular}
\end{center}
\end{lem}

\begin{proof}
It follows from a straightforward computation.
\end{proof}

\begin{lem}\label{lem1}
Suppose that $n$ is odd. Let $1\leqslant r\leqslant p^s$. Then we
have:

(1) The matrix $A_{r, i}$ is not similar to $A_{r, i}^{-1}$ for all
$1\leqslant i\leqslant t-1$.

(2) The matrix $A_{r, 0}$ is similar to $A_{r, 0}^{-1}$.
\end{lem}

\begin{proof}
(1) It follows from the fact that the eigenvalues of the two
matrices are not equal.

(2) Let $T=(t_{ij})\in M_r(k)$ be an upper triangular matrix defined
by $t_{rr}=1$ and the relations $t_{ij}=-t_{i+1,j}-t_{i+1,j+1}$ and
$t_{ii}=-t_{i+1,i+1}$ for all $1\leqslant i<j$. Then it is easy to
check that $A_{r, 0}TA_{r, 0}=T$.
\end{proof}

\begin{lem}\label{label2}
Suppose that $n$ is even. Let $1\leqslant r\leqslant p^s$ and
$0\leqslant i\leqslant t-1$. Then we have:

(1) The matrix $A_{r, i}$ is not similar to $A_{r, i}^{-1}$ if $i\neq 0$, $\frac{t}{2}$.

(2) The matrix $A_{r, i}$ is similar to $A_{r, i}^{-1}$ if $i=0$ or $\frac{t}{2}$.
\end{lem}

\begin{proof}
(1) It is similar to the proof of Lemma \ref{lem1}(1).

(2) When $i=0$, the proof is similar to that of Lemma \ref{lem1}(2).
When $i=\frac{t}{2}$, let $T_1=(t_{ij})\in M_r(k)$ be an upper
triangular matrix defined by $t_{rr}=1$ and the relations
$t_{ij}=-\xi^{\frac{t}{2}}t_{i+1,j}-t_{i+1,j+1}$ and
$t_{ii}=-t_{i+1,i+1}$ for all $1\leqslant i<j$. Then one can easily
check that $A_{r, \frac{t}{2}}T_1A_{r, \frac{t}{2}}=T_1$.
\end{proof}

\begin{lem}\label{T2=T}
Let $T$ and $T_1$ be the matrices defined respectively in the proof
of Lemma \ref{lem1} and Lemma \ref{label2}. Then $T^2=I$ and
$T_1^2=I$, where $I$ is an identity matrix.
\end{lem}

\begin{proof}
We only check that $T^2=I$ since the proof for $T_1^2=I$ is similar.
Let $T^2=(b_{il})_{r\times r}$. Clearly,
$T^2$ is an upper triangular matrix and $b_{ii}=1$. We use induction
on rows to show that $b_{il}=0$ for all $i<l$.

At first, we have
$$b_{r-1, r}=t_{r-1, r-1}t_{r-1, r}+t_{r-1, r}t_{r,
r}= t_{r-1, r}(t_{r-1, r-1}+t_{r, r})=0.$$
Next, let $i<r-1$ and suppose that $b_{jl}=0$ for all $i<j< l\leqslant r$.
Then we have $b_{i,i+1}=t_{i,i}t_{i,i+1}+t_{i,i+1}t_{i+1,i+1}=0$, and
for $i+1<l$,

\begin{eqnarray*}
b_{il}&=&\sum_{k=i}^lt_{ik}t_{kl}=t_{ii}t_{il}+\sum_{k=i+1}^lt_{ik}t_{kl}=t_{ii}t_{il}+\sum_{k=i+1}^l(-t_{i+1, k}-t_{i+1, k+1})t_{kl}\\
&=&t_{ii}t_{il}-\sum_{k=i+1}^lt_{i+1, k}t_{kl}-\sum_{k=i+1}^lt_{i+1,
k+1}t_{kl}=t_{ii}t_{il}-b_{i+1, l}-\sum_{k=i+1}^lt_{i+1, k+1}t_{kl}\\
&=&t_{ii}t_{il}-\sum_{k=i+1}^lt_{i+1, k+1}(-t_{k+1, l}-t_{k+1, l+1})\\
&=&t_{ii}t_{il}+\sum_{k=i+1}^lt_{i+1, k+1}t_{k+1, l}+\sum_{k=i+1}^lt_{i+1, k+1}t_{k+1, l+1}\\
&=&t_{ii}t_{il}+\sum_{k=i+1}^lt_{i+1, k}t_{k, l}-t_{i+1, i+1}t_{i+1,
l}
+\sum_{k=i+1}^{l+1}t_{i+1, k}t_{k, l+1}-t_{i+1, i+1}t_{i+1, l+1}\\
&=&t_{ii}t_{il}+b_{i+1, l}-t_{i+1, i+1}t_{i+1, l}
+b_{i+1, l+1}-t_{i+1, i+1}t_{i+1, l+1}\\
&=&t_{ii}t_{il}-t_{i+1, i+1}t_{i+1, l}
-t_{i+1, i+1}t_{i+1, l+1}\\
&=&-t_{i+1, i+1}(t_{i l}+t_{i+1, l}+t_{i+1, l+1})=0.
\end{eqnarray*}
This shows that $T^2=I$.
\end{proof}

\begin{lem}\cite{huppert}\label{quention}
If $H\lhd G$ with $G$ a finite group and $W$ is a $kG$-module, then
$$W\downarrow_H\uparrow^G\cong k(G/H)\otimes W$$
where the $kG$-module structure of $k(G/H)$ is given by $g_2\cdot
(g_1H)=g_2g_1H$.
\end{lem}

\begin{theo}\cite{willems}\label{willems}
Let $H$ be a normal subgroup of a finite group $G$ such that $G/H$ is a cyclic
$p'$-group. Let $V$ be an indecomposable $kH$-module with the
inertia group $T(V)=G$. Then there exists a $kG$-module $W$ such
that $W\!\!\downarrow_H\cong V$.
\end{theo}

\begin{lem}\label{direct}
As a $kD_n$-module, $k(D_n/C_n)$ is a direct sum of $2$ submodules
of dimensions $1$.
\end{lem}
\begin{proof}
Obviously, there is a $k$-basis $\{v_1, v_2\}$ of $k(D_n/C_n)$ such
that the action of $kD_n$ on $k(D_n/C_n)$ is given by
\begin{equation*}
\begin{cases}a\cdot v_1=v_1,   \\ a\cdot v_2=v_2,
\end{cases}
\begin{cases}b\cdot v_1=v_2,   \\ b\cdot v_2=v_1.
\end{cases}
\end{equation*}
Hence, $V_1={\rm span}\{v_1+v_2\}$ and $V_2={\rm span}\{v_1-v_2\}$ are
submodules of $k(D_n/C_n)$, and $k(D_n/C_n)=V_1\oplus V_2$.
\end{proof}

Let $V_{r, i}$ be the indecomposable $kC_n$-module corresponding to
the representation $\rho_{r, i}$, where $1\leqslant r\leqslant p^s$
and $0\leqslant i\leqslant t-1$. For any $x\in D_n$, let $${}^xV_{r,
i}=x\otimes V_{r, i}=\{x\otimes v|v\in V_{r, i}\}.$$ Then ${}^xV_{r,
i}$ is a $kC_n$-module with the action given by $a(x\otimes
v)=x\otimes yv$, where $a\in C_n$ and $y=x^{-1}ax\in
x^{-1}C_nx=C_n$. Let $T(V_{r, i})=\{x\in D_n|x\otimes V_{r, i}\cong
V_{r, i}\}$ be the inertia group of $V_{r, i}$.

\begin{theo}\label{induced}
Suppose that $n$ is odd. Let $1\leqslant r\leqslant p^s$. Then we have

(1) The representation $\Omega_{2r, i}$ induced from $\rho_{r, i}$ is
indecomposable for any $1\leqslant i\leqslant
t-1$.

(2) The representation $\Omega_{2r, 0}$ induced from $\rho_{r, 0}$ is
a direct sum of two indecomposable representations $\Phi_{r, 0}$ and
$\Phi_{r,0}'$ defined by
$$\Phi_{r, 0}(a)=A_{r, 0},\ \Phi_{r, 0}(b)=T \mbox{ and }
\Phi_{r, 0}'(a)=A_{r, 0},\ \Phi_{r, 0}'(b)=-T,$$
respectively, where $T$ is
the matrix defined in the proof of Lemma \ref{lem1}.
\end{theo}

\begin{proof}
(1) We consider the inertia group $T(V_{r, i})$ of $V_{r, i}$.
Clearly, $C_n\subseteq T(V_{r, i})$.  Let ${}^{ba^j}\!\!\rho_{r, i}$
be the representation corresponding to the $kC_n$-module
$^{ba^j}\!V_{r, i}$, $0\leqslant j\leqslant n-1$. Then
${}^{ba^j}\!\!\rho_{r, i}(a)=A_{r, i}^{-1}$, which is not similar to
$\rho_{r, i}(a)=A_{r, i}$ by Lemma \ref{lem1}. Hence $T(V_{r,
i})=C_n$, and the result follows from Theorem \ref{T(W)}.

(2) By Lemma \ref{lem1}, we have $\rho_{r, 0}({}^b\!\!a)=\rho_{r,
0}(bab^{-1})= \rho_{r, 0}(a^{-1})=T\rho_{r, 0}(a)T^{-1}. $ Clearly,
$\rho_{r, 0}({}^{b^i}\!\!a^j)=T^i\rho_{r, 0}(a^j)T^{-i}$ for any
integers $i$ and $j$.

Following the method described in \cite{willems} or \cite[section
9]{huppert}, we define $\Phi_{r, 0}$ by putting $\Phi_{r,
0}(xb^i)=\rho_{r, 0}(x)T^i$ for all $x\in C_n$,  where $i=0, 1$.
Then $\Phi_{r, 0}(xb^j)=\rho_{r, 0}(x)T^j$ for any integer $j$ since
$T^2=I$ by Lemma \ref{T2=T}. Thus for any positive integers $i$, $j$, and
$x, y\in C_n$, we have
\begin{eqnarray*}
\Phi_{r, 0}[(xb^i)(yb^j)]&=&\Phi_{r, 0}[(xb^iyb^{-i})b^{i+j}]
=\Phi_{r, 0}[(x\:{}^{b^i}\!\!y)b^{i+j}]=\rho_{r, 0}(x\:{}^{b^i}\!\!y)T^{i+j}\\
&=&\rho_{r, 0}(x)\rho_{r, 0}({}^{b^i}\!\!y)T^{i+j} =\rho_{r,
0}(x)T^i\rho_{r, 0}(y)T^j=\Phi_{r, 0}(xb^i)\Phi_{r, 0}(yb^j).
\end{eqnarray*}
Hence $\Phi_{r, 0}$ is an indecomposable representation of $kD_n$
and $\Phi_{r, 0}(a)=A_{r, 0},\  \Phi_{r, 0}(b)=T.$

Now let $W$ be the indecomposable $kD_n$-module corresponding to the
representation $\Phi_{r, 0}$. Then clearly $V_{r, 0}\cong
W\!\!\downarrow_{C_n}$. The existence of such a $kD_n$-module $W$
also follows from Theorem \ref{willems}. By Lemma \ref{lem1} and the
proof of Part (1), $T(V_{r, 0})=D_n$. Hence, $V_{r,
0}\!\!\uparrow^{D_n}\cong
W\!\!\downarrow_{C_n}\!\!\uparrow^{D_n}\cong k(D_n/C_n)\otimes
W\cong V_1\otimes W\oplus V_2\otimes W$ by Lemma \ref{quention} and
Lemma \ref{direct}. In this case, $V_1\otimes W\cong W$ and the
representation $\Phi_{r, 0}'$ corresponding to the $kD_n$-module
$V_2\otimes W$ is given by $\Phi_{r, 0}'(a)=A_{r, 0},\  \Phi_{r,
0}'(b)=-T.$
\end{proof}

When $n$ is even, we have the following similar result.

\begin{theo}\label{Theven}
Suppose that $n$ is even. Let $1\leqslant r\leqslant p^s$ and
$0\leqslant i\leqslant t-1$.

(1) If $i\neq0, \frac{t}{2}$, then the representation $\Omega_{2r, i}$ induced from $\rho_{r, i}$ is
indecomposable.

(2) If $i=0$ or $\frac{t}{2}$, then the representation $\Omega_{2r, i}$ induced from $\rho_{r, i}$
is a direct sum of two indecomposable
representations $\Phi_{r, i}$ and $\Phi_{r, i}'$ defined by
$$\Phi_{r, 0}(a)=A_{r, 0},\  \Phi_{r, 0}(b)=T;\
\Phi_{r, 0}'(a)=A_{r, 0},\  \Phi_{r, 0}'(b)=-T;$$
$$\Phi_{r, \frac{t}{2}}(a)=A_{r, \frac{t}{2}},\  \Phi_{r,
\frac{t}{2}}(b)=T_1;\ \Phi_{r, \frac{t}{2}}'(a)=A_{r,
\frac{t}{2}},\  \Phi_{r, \frac{t}{2}}'(b)=-T_1,$$
respectively, where $T$ and
$T_1$ are the matrices defined in the proofs of Lemma \ref{lem1} and
Lemma \ref{label2}, respectively.
\end{theo}

\begin{cor}
Suppose that $n$ is odd. Then $\left\{\Phi_{1, 0},\ \Phi_{1, 0}',\
\Omega_{2, i}\left|1\leqslant i\leqslant
\frac{t-1}{2}\right.\right\}$ is a complete set of irreducible
representations of $kD_n$ up to isomorphism.
\end{cor}

\begin{proof}
It is easy to show that $\Omega_{2, i}$ is irreducible for any
$1\leqslant i\leqslant \frac{t-1}{2}$. Then the
lemma follows from the following well known theorem.
\end{proof}

\begin{theo}[R. Brauer]
Let $k$ be a splitting field of characteristic $p$ for a finite
group $G$. The number of isomorphism classes of simple $kG$-modules
is equal to the number of conjugate classes of $G$ which consist of
$p$-regular elements.
\end{theo}

When $n$ is even, $kD_n$ has $\frac{t+6}{2}$ irreducible
representations.
\begin{cor}
Suppose that $n$ is even. Then $\left\{\Phi_{1, 0},\ \Phi_{1, 0}',\
\Phi_{1, \frac{t}{2}},\ \Phi_{1, \frac{t}{2}}',\ \Omega_{2,
i}\left|1\leqslant i\leqslant\frac{t}{2}-1\right.\right\}$ is a
complete set of irreducible representations of $kD_n$ up to
isomorphism.
\end{cor}

\begin{lem}\label{x}
The following upper triangular matrix is invertible
\begin{center}
\begin{tabular}{c}
 $X=\left(\begin{tabular}{cccccccc}
$1$&$1$&$1$&$1$&$\cdots$&$1$&$1$&$1$\\
$0$&$x_{22}$&$x_{23}$&$x_{24}$&$\cdots$&$x_{2, r-2}$&$x_{2, r-1}$&$x_{2, r}$\\
$0$&$0$&$x_{33}$&$x_{34}$&$\cdots$&$x_{3, r-2}$&$x_{3, r-1}$&$x_{3, r}$\\
$\cdots$&$\cdots$&$\cdots$&$\cdots$&$\cdots$&$\cdots$&$\cdots$&$\cdots$\\
$0$&$0$&$0$&$0$&$\cdots$&$x_{r-2, r-2}$&$x_{r-2, r-1}$&$x_{r-2, r}$\\
$0$&$0$&$0$&$0$&$\cdots$&$0$&$x_{r-1, r-1}$&$x_{r-1, r}$\\
$0$&$0$&$0$&$0$&$\cdots$&$0$&$0$&$x_{r, r}$
 \end{tabular}\right), $
 \end{tabular}
\end{center}
where $x_{2k}=-\xi^{2i}\sum_{y=0}^{k-2}(-\xi^i)^y$ and
$x_{jk}=\sum_{y=j-1}^{k-1}(-1)^{k-y}\xi^{(k+1-y)i}x_{(j-1)y}$ for
all $3\leqslant j\leqslant k$.
\end{lem}

\begin{proof}
The determinant of $X$ is
$\Pi_{y=2}^{r}x_{yy}=(-1)^z\xi^{r(r-1)i}\neq 0$ for some $z\in
\mathbb{Z^+}$.
\end{proof}

\begin{theo}\label{odd}
Suppose that $n$ is odd. Let $1\leqslant i\leqslant\frac{t-1}{2}$ and
$\frac{t+1}{2}\leqslant j\leqslant t-1$. Then the representations $\Omega_{2r,
i}$ and $\Omega_{2r, j}$ are isomorphic if and only if $i+j=t$. In
this case, the transformation matrix is
\begin{tabular}{l}
$\left(\begin{tabular}{cc}
$0$&$X$\\
$X$&$0$
 \end{tabular}\right)$, where $X$ is the matrix defined in Lemma
 \ref{x}.
  \end{tabular}
\end{theo}

\begin{proof}
It suffices to show that
\begin{center}
\begin{tabular}{c}
$\left(\begin{tabular}{cc}
$A_{r, i}$&$0$\\
$0$&$A_{r, i}^{-1}$
 \end{tabular}\right)$
$\left(\begin{tabular}{cc}
$0$&$X$\\
$X$&$0$
 \end{tabular}\right)=$
$\left(\begin{tabular}{cc}
$0$&$X$\\
$X$&$0$
 \end{tabular}\right)$
$\left(\begin{tabular}{cc}
$A_{r, t-i}$&$0$\\
$0$&$A_{r, t-i}^{-1}$
 \end{tabular}\right), $
  \end{tabular}
  \end{center}
which is equivalent to $A_{r, i}X=XA_{r, t-i}^{-1}$.

Let $A_{r, i}X=(c_{vu})_{r\times r}$ and $XA_{r,
t-i}^{-1}=(d_{vu})_{r\times r}$. We need to check $c_{vu}=d_{vu}$ for
all $1\leqslant v, u\leqslant r$. In fact, if $v=u=1$, then $c_{11}=d_{11}=\xi^i$.
If $v=1$ and $u\neq1$, then
$c_{1u}=\xi^i+x_{2u}=\xi^i-\xi^{2i}+\cdots+(-1)^{u+1}\xi^{ui}=d_{1u}$.
If $v>u$, then $c_{vu}=d_{vu}=0$.
If $v=u\neq1$, then $c_{vv}=d_{vv}=\xi^ix_{vv}$.
If $v\neq1$ and $u=v+j$ for $1\leqslant j\leqslant r-v$, then
$c_{vu}=\xi^ix_{v(v+j)}+x_{(v+1)(v+j)}=$
$\sum_{y=v}^{v+j-1}(-1)^{v+j-y}\xi^{(v+j+1-y)i}x_{vy}=d_{vu}$. This
completes the proof.
\end{proof}

\begin{theo}
Suppose that $n$ is even. Let $1\leqslant i\leqslant\frac{t}{2}-1$ and
$\frac{t}{2}+1\leqslant j\leqslant t-1$. Then the representations $\Omega_{2r,
i}$ and $\Omega_{2r, j}$ are isomorphic if and only if $i+j=t$. In
this case, the transformation matrix is
\begin{tabular}{l}
$\left(\begin{tabular}{cc}
$0$&$X$\\
$X$&$0$
\end{tabular}\right)$,
\end{tabular}
where $X$ is the matrix defined in Lemma \ref{x}.
\end{theo}
\begin{proof}
It is similar to the proof of Theorem \ref{odd}.
\end{proof}

We summarize the above discussion as follows.

\begin{theo}
(1) Suppose that $n=p^st$ is odd. Then $\left\{\Phi_{r, 0},\
\Phi_{r, 0}',\ \Omega_{2r, i}\left|1\leqslant r\leqslant p^s,
1\leqslant i\leqslant\frac{t-1}{2}\right.\right\}$ is a complete set
of indecomposable representations of $kD_n$ up to isomorphism.

(2) Suppose that $n=p^st$ is even. Then $\left\{\Phi_{r, 0},\
\Phi_{r, 0}',\ \Phi_{r, \frac{t}{2}},\ \Phi_{r, \frac{t}{2}}',\
\Omega_{2r, i}\left|1\leqslant r\leqslant p^s, 1\leqslant
i\leqslant\frac{t}{2}-1\right.\right\}$ is a complete set of
indecomposable representations of $kD_n$ up to isomorphism.
\end{theo}

\section{The indecomposable representations of $D(kD_n)$}

In this section we will classify all indecomposable
$D(kD_n)$-modules up to isomorphism by means of Yetter-Drinfeld
$kD_n$-modules. We will frequently use Theorem \ref{main1}, Theorem \ref{main2}
and Corollary \ref{main3}, but not explicitly mention them for
the sake of simplicity.

From Section 3, when $n=p^st$ is odd, the representatives of
conjugate classes of $D_n$ are $1, b$ and $a^i$ $(1\leqslant
i\leqslant \frac{n-1}{2})$. When $n=p^st$ is even, the
representatives of conjugate classes of $D_n$ are $1, b, ab$ and
$a^i$ $(1\leqslant i\leqslant \frac{n}{2})$.

Assume that $n=p^st$ is odd. Let $V_{r, 0}$, $V_{r, 0}'$ and $W_{2r,
i}$ be the indecomposable $kD_n$-modules corresponding to the
representations $\Phi_{r, 0}$, $\Phi_{r, 0}'$ and $\Omega_{2r, i}$,
respectively, where $1\leqslant r\leqslant p^s$ and $1\leqslant
i\leqslant \frac{t-1}{2}$. Since the centralizer $C_{D_n}(1)=D_n$,
$V_{r, 0}$, $V_{r, 0}'$ and $W_{2r, i}$ are non-isomorphic
indecomposable Yetter-Drinfeld $kD_n$-modules, by Theorem
\ref{main1}, Theorem \ref{main2} and Corollary \ref{main3}, with the
comodule structures given by $\varphi:V\rightarrow V\otimes kD_n,\
\varphi(v)=v\otimes 1,\ v\in V,$ where $V=V_{r, 0}$, $V_{r, 0}'$ or
$W_{2r, i}$.

$C_{D_n}(b)=\langle b\rangle$ is a cyclic group of order $2$. The
representatives set of left cosets of $C_{D_n}(b)$ in $D_n$ is
$C_n$. Let $U_1$ and $U_2$ be the $2$ simple 1-dimensional
$kC_{D_n}(b)$-modules given by $b\cdot u=(-1)^ju,\ u\in U_j,\
1\leqslant j\leqslant 2.$ Then the induced $kD_n$-modules
$D(U_1)=kD_n\otimes_{kC_{D_n}(b)}U_1$ and
$D(U_2)=kD_n\otimes_{kC_{D_n}(b)}U_2$ are non-isomorphic simple
Yetter-Drinfeld $kD_n$-modules of dimension $n$, with the comodule
structures given by $\varphi:D(U_j)\rightarrow D(U_j)\otimes kD_n,\
\varphi(a^i\otimes u)= (a^i\otimes u)\otimes a^{2i}b,$ where
$0\leqslant i \leqslant n-1$, $u\in U_j$.

$C_{D_n}(a^l)=C_n$ $(1\leqslant l\leqslant\frac{n-1}{2})$ is a cyclic group of
order $n$. The representatives set of left cosets of $C_{D_n}(a^l)$
in $D_n$ is $\{1, b\}$. Let $Q_{r, j}$ be the indecomposable
$kC_n$-module corresponding to the representation $\rho_{r, j}$
defined in Theorem \ref{prrep}, where $1\leqslant r\leqslant p^s$ and $0\leqslant
j\leqslant t-1$. For any $l$ with $1\leqslant l\leqslant\frac{n-1}{2}$, let $Q_{r,
j}^l=Q_{r, j}$ as $kC_n$-modules. Then the induced $kD_n$-modules
$D(Q_{r, j}^l)=kD_n\otimes_{kC_n}Q_{r, j}^l$ are non-isomorphic
indecomposable Yetter-Drinfeld $kD_n$-modules of dimension $2r$,
with the comodule structures given by
$$\varphi:D(Q_{r, j}^l)\rightarrow D(Q_{r, j}^l)\otimes kD_n, $$
$$\varphi(1\otimes
q)=(1\otimes q)\otimes a^l,\ \ \varphi(b\otimes q)=(b\otimes q)\otimes a^{n-l},$$
where $q\in Q_{r, j}^l$, $1\leqslant r\leqslant
p^s$, $0\leqslant j\leqslant t-1$ and $1\leqslant l\leqslant\frac{n-1}{2}$.

We summarize the above discussion as follows.

\begin{theo}
Assume that $n=p^st$ is odd. Then
$$\left\{\begin{array}{c}
V_{r, 0},\ V_{r, 0}',\ W_{2r, i},\\
D(U_1),\ D(U_2),\ D(Q_{r, j}^l)\\
\end{array}\left|\begin{array}{c}
1\leqslant r\leqslant p^s, 0\leqslant j\leqslant t-1,\\
1\leqslant i\leqslant\frac{t-1}{2}, 1\leqslant
l\leqslant\frac{n-1}{2}\\
\end{array}\right.\right\}$$ is a complete set of indecomposable
Yetter-Drinfeld $kD_n$-modules up to isomorphism.
\end{theo}

Now assume that $n$ is even. Let $V_{r, 0}$, $V_{r, 0}'$, $V_{r,
\frac{t}{2}}$, $V_{r, \frac{t}{2}}'$ and $W_{2r, i}$ be the
indecomposable $kD_n$-modules corresponding to the representations
$\Phi_{r, 0}$, $\Phi_{r, 0}'$, $\Phi_{r, \frac{t}{2}}$, $\Phi_{r,
\frac{t}{2}}'$ and $\Omega_{2r, i}$ given in Theorem \ref{Theven},
respectively, where $1\leqslant r\leqslant p^s$ and $1\leqslant i\leqslant
\frac{t}{2}-1$. Then the $kD_n$-modules $V_{r, 0}$, $V_{r, 0}'$,
$V_{r, \frac{t}{2}}$, $V_{r, \frac{t}{2}}'$ and $W_{2r, i}$ are
non-isomorphic indecomposable Yetter-Drinfeld $kD_n$-modules with
the comodule structures given by
$$\varphi:V\rightarrow V\otimes kD_n,\  \varphi(v)=v\otimes 1,$$
where $V=V_{r, 0}$, $V_{r, 0}'$, $V_{r, \frac{t}{2}}$, $V_{r,
\frac{t}{2}}'$ or $W_{2r, i}$, and $1\leqslant r\leqslant p^s$, $1\leqslant
i\leqslant\frac{t}{2}-1$.

$C_{D_n}(b)=\langle b,a^{\frac{n}{2}}\rangle$ is a Klein four group
generated by $b$ and $a^{\frac{n}{2}}$.  The left coset
representatives of $C_{D_n}(b)$ in $D_n$ are $a^j$, $0\leqslant j\leqslant
\frac{n}{2}-1$. Let $V_1, V_2, V_3, V_4$ be the only $4$ simple
$kC_{D_n}(b)$-modules, which are all 1-dimensional and defined by
$$\left\{\begin{array}{l}
a^{\frac{n}{2}}\cdot v_1=v_1 ,\\
\ \ b\cdot v_1=v_1 ,\\
\end{array}
\right.\ \ \left\{\begin{array}{l}
a^{\frac{n}{2}}\cdot v_2=v_2 ,\\
\ \ b\cdot v_2=-v_2 ,\\
\end{array}
\right.\ \ \left\{\begin{array}{l}
a^{\frac{n}{2}}\cdot v_3=-v_3 ,\\
\ \ b\cdot v_3=v_3 ,\\
\end{array}
\right.\ \ \left\{\begin{array}{l}
a^{\frac{n}{2}}\cdot v_4=-v_4 ,\\
\ \ b\cdot v_4=-v_4 ,\\
\end{array}
\right.$$ where $v_i\in V_i$, $1\leqslant i\leqslant 4$. Then the
induced $kD_n$-modules $D(V_i)=kD_n\otimes_{C_{D_n}(b)}V_i$ are
non-isomorphic simple Yetter-Drinfeld $kD_n$-modules of dimension
$\frac{n}{2}$, with the comodule structures given by
$\varphi:D(V_i)\rightarrow D(V_i)\otimes kD_n,\ \varphi(a^j\otimes
v_i)=(a^j\otimes v_i)\otimes a^{2j}b,$ where $v_i\in V_i$,
$1\leqslant i\leqslant 4$ and $0\leqslant j\leqslant \frac{n}{2}-1$.

Similarly, $C_{D_n}(ab)=\langle ab, a^{\frac{n}{2}}\rangle$ is a
Klein four group generated by $ab$ and $a^{\frac{n}{2}}$. The left
coset representatives of $C_{D_n}(b)$ in $D_n$ are $a^j$, $0\leqslant
j\leqslant \frac{n}{2}-1$. Let $U_1, U_2, U_3, U_4$ be the only $4$
simple $kC_{D_n}(ab)$-modules, which are all 1-dimensional and
defined by
$$\left\{\begin{array}{l}
a^{\frac{n}{2}}\cdot u_1=u_1 ,\\
ab\cdot u_1=u_1 ,\\
\end{array}
\right.\ \ \left\{\begin{array}{l}
a^{\frac{n}{2}}\cdot u_2=u_2 ,\\
ab\cdot u_2=-u_2 ,\\
\end{array}
\right.\ \ \left\{\begin{array}{l}
a^{\frac{n}{2}}\cdot u_3=-u_3 ,\\
ab\cdot u_3=u_3 ,\\
\end{array}
\right.\ \ \left\{\begin{array}{l}
a^{\frac{n}{2}}\cdot u_4=-u_4 ,\\
ab\cdot u_4=-u_4 ,\\
\end{array}
\right.$$ where $u_i\in U_i$, $1\leqslant i\leqslant 4$. Then the
induced $kD_n$-modules $D(U_i)=kD_n\otimes_{C_{D_n}(ab)}U_i$ are
non-isomorphic simple Yetter-Drinfeld $kD_n$-modules of dimension
$\frac{n}{2}$, with the comodule structures given by
$\varphi:D(U_i)\rightarrow D(U_i)\otimes kD_n,\ \varphi(a^j\otimes
u_i)=(a^j\otimes u_i)\otimes a^{2j+1}b,$ where $u_i\in U_i$,
$1\leqslant i\leqslant 4$ and $0\leqslant j\leqslant \frac{n}{2}-1$.

$C_{D_n}(a^l)=C_n$ $(1\leqslant l\leqslant\frac{n}{2}-1)$ is a cyclic group of
order $n$. The left coset representatives of $C_{D_n}(a^l)$ in $D_n$
are $1, b$. Let $Q_{r, j}$ be the indecomposable $kC_n$-module
corresponding to the representation $\rho_{r, j}$ defined in Theorem
\ref{prrep} for any $1\leqslant r\leqslant p^s$ and $0\leqslant j\leqslant t-1$. For any
$l$ with $1\leqslant l\leqslant\frac{n}{2}-1$, let $Q_{r, j}^l=Q_{r, j}$ as
$kC_n$-modules. Then the induced $kD_n$-modules $D(Q_{r,
j}^l)=kD_n\otimes_{kC_n}Q_{r, j}^l$ are non-isomorphic
indecomposable Yetter-Drinfeld $kD_n$-modules of dimension $2r$,
with the comodule structures given by
$$\varphi:D(Q_{r, j}^l)\rightarrow D(Q_{r, j}^l)\otimes kD_n, $$
$$\varphi(1\otimes
q)=(1\otimes q)\otimes a^l,\ \ \varphi(b\otimes q)=(b\otimes q)\otimes a^{n-l},$$
where $q\in Q_{r, j}^l$, $1\leqslant r\leqslant
p^s$, $0\leqslant j\leqslant t-1$ and $1\leqslant l\leqslant\frac{n}{2}-1$.

$C_{D_n}(a^{\frac{n}{2}})=D_n$.  Let $V_{r, 0}, V_{r, 0}', V_{r,
\frac{t}{2}}, V_{r, \frac{t}{2}}'$ and $W_{2r, i}$ be the
indecomposable $kD_n$-modules corresponding to the representations
$\Phi_{r, 0}, \Phi_{r, 0}'$, $\Phi_{r, \frac{t}{2}}, \Phi_{r,
\frac{t}{2}}'$ and $\Omega_{2r, i}$, respectively, where $1\leqslant r\leqslant p^s$
and $1\leqslant i\leqslant \frac{t}{2}-1$. Then $D(V_{r, 0}),
D(V_{r, 0}')$,  $D(V_{r, \frac{t}{2}}), D(V_{r, \frac{t}{2}}')$ and
$D(W_{2r, i})$ are non-isomorphic indecomposable Yetter-Drinfeld
$kD_n$-modules, where $D(V_{r, 0})=V_{r, 0}$, $D(V_{r, 0}')=V_{r,
0}'$,  $D(V_{r, \frac{t}{2}})=V_{r, \frac{t}{2}}$,  $D(V_{r,
\frac{t}{2}}')=V_{r, \frac{t}{2}}'$ and $D(W_{2r, i})=W_{2r, i}$ as
$kD_n$-modules.  The comodule structures are given by
$$\varphi:D(V)\rightarrow D(V)\otimes kD_n,\ \varphi(v)=v\otimes a^{\frac{n}{2}},$$
for all $v\in V$,  where $V=V_{r, 0}, V_{r, 0}', V_{r,
\frac{t}{2}}, V_{r, \frac{t}{2}}'$ or $W_{2r, i}$, and $1\leqslant r\leqslant
p^s$, $1\leqslant i\leqslant \frac{t}{2}-1$.

We summarize the above discussion as follows.

\begin{theo}
Assume that $n=p^st$ is even. Then
$$\left\{\left.\begin{array}{c}
V_{r, 0},\ V_{r, 0}',\ V_{r, \frac{t}{2}},\ V_{r, \frac{t}{2}}',\ W_{2r,
i},\\
D(V_{r, 0}),\ D(V_{r, 0}'),\ D(V_{r, \frac{t}{2}}),\ D(V_{r,
\frac{t}{2}}'),\\
D(W_{2r, i}),\ D(V_m),\ D(U_m),\ D(Q_{r, j}^l)\\
\end{array}\right|\begin{array}{l}
1\leqslant r\leqslant p^s, 1\leqslant i\leqslant\frac{t}{2}-1,\\
0\leqslant j\leqslant t-1, 1\leqslant m \leqslant 4\\
1\leqslant l\leqslant\frac{n}{2}-1\\
\end{array}\right\}$$ is a complete set of indecomposable Yetter-Drinfeld
$kD_n$-modules up to isomorphism.
\end{theo}

Thus we have classified all finite dimensional indecomposable Yetter-Drinfeld $kD_n$-
modules up to isomorphism.

\end{document}